\documentclass[10pt]{amsart}\oddsidemargin -1mm \evensidemargin -1mm \topmargin -15mm\headheight5mm \headsep 3.5mm \textheight 245mm\textwidth 170mm 
\usepackage{amsfonts, amssymb, amsthm, amsmath, euscript, hhline}
\theoremstyle{definition}\newtheorem{theorem}{Theorem}[section]\newtheorem{lemma}[theorem]{Lemma}\newtheorem{proposition}[theorem]{Proposition}\newtheorem{corollary}[theorem]{Corollary}\newtheorem{definition}[theorem]{Definition}\newtheorem{Ex}[theorem]{Example}
\DeclareMathOperator{\Ue}{\mathrm{U}}\DeclareMathOperator{\gr}{\mathrm{gr}}\DeclareMathOperator{\Sa}{\mathrm{S}}\DeclareMathOperator{\Var}{\mathrm{Var}}\DeclareMathOperator{\Ann}{\mathrm{Ann}}\DeclareMathOperator{\GVar}{\mathrm{GVar}}

\newcounter{AP}
\begin{document}
\author{Alexey Petukhov}\address{Alexey Petukhov:
Jacobs University Bremen, Campus Ring 1, Bremen D28759, Germany\\ on leave from the Institute for Information Transmission Problems, Bolshoy Karetniy 19-1, Moscow 127994, Russia}
\email{alex-{}-2@yandex.ru}

\title{On annihilators of bounded $(\frak g, \frak k)$-modules}\maketitle
\begin{abstract} Let $\frak g$ be a semisimple Lie algebra and $\frak k\subset\frak g$ be a reductive subalgebra. We say that a $\frak g$-module $M$ is a {\it bounded $(\frak g, \frak k)$-module}  if $M$ is a direct sum of simple finite-dimensional $\frak k$-modules and the multiplicities of all simple $\frak k$-modules in that direct sum are universally bounded.

The goal of this article is to show that the ``boundedness'' property for a simple $(\frak g, \frak k)$-module $M$ is equivalent to a property of the associated variety of the annihilator of $M$ (this is the closure of a nilpotent coadjoint orbit inside $\frak g^*$) under the assumption that the main field is algebraically closed and of characteristic 0. In particular this implies that if $M_1, M_2$ are simple $(\frak g, \frak k)$-modules such that $M_1$ is bounded  and the associated varieties of the annihilators of $M_1$ and $M_2$ coincide then $M_2$ is also bounded. This statement is a geometric analogue of a purely algebraic fact due to I.~Penkov and V.~Serganova and it was posed as a conjecture in my Ph.D. thesis.

{\bf Key words:} $(\frak g, \frak k)$-modules, spherical varieties, symplectic geometry.

{\bf AMS subject classification.} 13A50, 14L24, 17B08, 17B63, 22E47.
\end{abstract}
\section{Introduction} A notion of admissible $(\frak g, \frak k)$-module for a generic pair of Lie algebras $(\frak g, \frak k)$ is a quite straightforward generalization of both Harish-Chandra modules and $\frak g$-modules of category $\EuScript O$. Bounded $(\frak g, \frak k)$-modules~\cite{PS2} are a very specific subclass of the admissible $(\frak g, \frak k)$-modules and it turns out that this class is related to the notion of spherical variety~\cite{T}. The structure of this subcategory is still very misterious in general and few examples are evaluated explicitly, see~\cite{SG},~\cite{PSZ2},~\cite{P4}.

The set works on both Harish-Chandra modules and category $\EuScript O$ seems to be infinite and one can find some introduction to these subjects in~\cite{KV, H}. We would like to mention explicitly works~\cite{P1, P2, PZ, SG, PSZ1, Z, F, Ma, Mi}.

The goal of this article is to show that this ``boundedness'' property for a simple $(\frak g, \frak k)$-module $M$ is equivalent to a property of the associated variety of the annihilator of $M$ (this is the closure of a nilpotent coadjoint orbit inside $\frak g^*$) under the assumption that the main field is algebraically closed and of characteristic 0. Roughly speaking, a simple $(\frak g, \frak k)$-module $M$ is bounded if and only if the associated variety of the annihilator of $M$ is $K$-coisotropic with respect to the action of an algebraic group $K$ attached to the pair $(\frak g, \frak k)$.

In particular this implies that if $M_1, M_2$ are simple $(\frak g, \frak k)$-modules such that $M_1$ is bounded  and the associated varieties of the annihilators of $M_1$ and $M_2$ coincide then $M_2$ is also bounded. This statement is a geometric analogue of~\cite[Theorem~4.3]{PS1} and it was posed as a conjecture in my Ph.D. thesis~\cite[Conjecture 2.1]{P3}.
\section{Definitions}
\subsection{$(\frak g, \frak k)$-modules}Let $\frak g$ be a Lie algebra over an algebraically closed field $\mathbb F$ of characterstic 0 and $\frak k$ be a subalgebra of $\frak g$. We say that a $\frak g$-module $M$ is a {\it $(\frak g, \frak k)$-module} if
$$\forall m\in M (\dim(\Ue(\frak k)\cdot m)<\infty)$$
where $\Ue(\frak k)$ is the universal enveloping algebra of $\frak k$.

It is clear that any $(\frak g, \frak k)$-module is isomorphic to the direct limit of a directed set of finite-dimensional $\frak k$-modules. Thus for any given simple $\frak k$-module $V$ we can define the Jordan-H\"older multiplicity $[M: V]$. If such multiplicities $[M: V]$ are finite for all simple $\frak k$-modules $V$ then $M$ is called {\it admissible}. If there exists $C_M\in\mathbb N$ such that all multiplicities $[M: V]$ are bounded by $C_M$ then $M$ is called {\it bounded}.
\subsection{Associated varieties} Let $M$ be a finitely generated $\frak g$-module $M$ (for example $M$ can be a simple $\frak g$-module). We consider a finite-dimensional generating subspace $M_0$ of $M$ and define spaces $M_i~(i\ge0)$ inductively using formula$$M_{i+1}:=M_i+\frak g\cdot M_i.$$
The associated graded object $$\gr M:=\oplus_{i\ge0}M_{i+1}/M_i$$ is a module of the associated graded algebra $\Sa(\frak g)$ of ${\rm U}(\frak g)$. Let $J$ be the annihilator of $M$ in $\Sa(\frak g)$. We denote by $\Var(M)$ the set of points $\chi\in\frak g^*$ such that $f(\chi)=0$ for all $f\in J$.

It is well known that $\Var(M)$ is the same for all choices of $M_0$, see for example~\cite[p.~19]{P3}, and therefore ${\rm Var}(M)$ is a proper invariant of $M$. Next, we consider the annihilator $\Ann_{\Ue(\frak g)} M$ of $M$ in ${\rm U}(\frak g)$ as a left $\frak g$-module and put $$\GVar(M):=\Var(\Ann_{\Ue(\frak g)}M).$$
We denote by $\overline{G\cdot\Var(M)}$ the closure of the $G$-translation of $\Var(M)$. Note that in many interesting cases $$\GVar(M)=\overline{G\cdot\Var(M)}.$$

Put $\sqrt J:=\{f\in\Sa(\frak g)\mid \exists n\in\mathbb Z_{>0}: (f^n\in J)\}$. The ideal $J$ is {\it involutive}, i.e. $\{\sqrt J, \sqrt J\}\subset \sqrt J$ where $\{\cdot,~\cdot\}$ is the standard Poisson bracket on $\Sa(\frak g)$, see~\cite{G}.

Next, if $M$ is a $(\frak g, \frak k)$-module then one can choose filtration $\{M_i\}$ to be $\frak k$-stable and hence $\frak k\subset J$. This condition immediately implies $$\Var(M)\subset\frak k^\bot$$where $\frak k^\bot:=\{\chi\in\frak g^*\mid \chi|_{\frak k}=0\}.$
\section{Main theorem}\label{Smain}
We would use notions and terminology on symplectic geometry and algebraic groups actions of~\cite{TZh}. Fix a semisimple Lie algebra $\frak g$ together with the adjoint group $G$ of $\frak g$, and also fix a subalgebra $\frak k\subset \frak g$. Assume that $\frak k$ is {\it algebraically reductive} in $\frak g$, that is there exists a connected reductive algebraic subgroup $K$ of $G$ such that $\frak k\subset\frak g$ is the Lie algebra of $K\subset G$. If such a subgroup $K$ exists then it is unique.
\begin{theorem}\label{Tmain}Assume that $\frak g$ is a semisimple Lie algebra and $\frak k$ is an algebraically reductive subalgebra of $\frak g$. Then, for a simple $(\frak g, \frak k)$-module $M$, the following conditions are equivalent:

(1) $M$ is a bounded $(\frak g, \frak k)$-module,

(2) $\GVar(M)$ is $K$-coisotropic.\end{theorem}
It turns out that it is possible to reduce the case of any subalgebra $\frak k$ of a semisimple Lie algebra $\frak g$ to the case of algebraically reductive subalgebra, see Propositions~\ref{Pbhat},~\ref{Pbhatr}.
\section{Algebraic subalgebras}
Let $\frak k$ be a subalgebra of a semisimple Lie algebra $\frak g$, $G$ be the adjoint group of $\frak g$, and let $\hat K$ be the least algebraic subgroup of $G$ such that the Lie algebra of $\hat K$ contains $\frak k$. Denote by $\hat{\frak k}$ the Lie algebra of $\hat K$. Then

$\bullet$  $\hat K$ is connected and $\hat K$ normalizes $\frak k$,

$\bullet$ $\hat{\frak k}/\frak k$ is abelian and $[\hat{\frak k}, \hat{\frak k}]=[\frak k, \frak k]$,\\
see~\cite[Subsection~3.3.3]{VO}.
\begin{definition} A subalgebra $\frak k$ is called {\it algebraic in} $\frak g$ if $\frak k=\hat{\frak k}$. A subalgebra $\frak k$ is called {\it reductive in} $\frak g$ if $\hat K$ is reductive.\end{definition}
It is clear that if $\frak k$ is algebraically reductive in $\frak g$ then $\frak k$ is reductive in $\frak g$. The inverse statement is false.
\begin{Ex}\label{Ealglie}Let $\frak g\cong\frak{sl}(3)$ be the Lie algebra of traceless 3$\times 3$ matrices and let $\frak k$ be the one-dimensional subalgebra spanned by
$$\left(\begin{array}{ccc}\alpha&0&0\\0&\beta&0\\0&0&-(\alpha+\beta)\end{array}\right),$$
$\alpha, \beta\in\mathbb F$. If
$$\alpha, \beta\in\mathbb F\backslash\mathbb Q{\rm~and~}\frac\alpha\beta\in\mathbb F\backslash\mathbb Q$$
then $\hat K$ is an algebraic torus consisting of the diagonal $3\times 3$ matrices with determinant 1, and therefore $\hat{\frak k}$ is the space of traceless diagonal 3$\times3$ matrices. In particular, this implies that $\frak k$ is reductive in $\frak{sl}(3)$, but $\frak k$ is {\bf not} algebraically reductive in $\frak{sl}(3)$.\end{Ex}
The following lemmas show that $\hat K$-modules, $\hat{\frak k}$-modules and $\frak k$-modules are close to each other.
\begin{lemma}\label{Laev1}Let $V_1, V_2$ be $\hat K$-modules. Then the following conditions are equivalent.

(1) $V_1\cong V_2$,

(2) $(V_1)|_{\hat k}\cong (V_2)|_{\hat k}$,

(3) $(V_1)|_{\frak k}\cong (V_2)|_{\frak k}$.\end{lemma}
\begin{proof}Conditions (1) and (2) are equivalent for a connected group in characteristic 0. It is clear that (2) implies (3). We left to show that (3) implies (1).

Assume (3). Consider vector space ${\rm Hom}_\mathbb F(V_1, V_2)$. We have that ${\rm Hom}_\frak k(V_1, V_2)\ne0$. Fix
\begin{equation}0\ne\phi\in{\rm Hom}_\frak k(V_1, V_2)\label{Ev1v2}\end{equation}
such that $\phi$ defines the isomorphism between $V_1$ and $V_2$. Let $H$ be the stabilizer of $\phi$ in $K$. Condition (\ref{Ev1v2}) implies that $\frak k$ is a subalgebra of the Lie algebra of $H$. The definition of $\hat K$ implies that $H=\hat K$. Therefore $\phi$ is an isomorphism of $\hat K$-modules.\end{proof}
\begin{lemma}\label{Laev2}Let $V$ be a $K$-module and let $W\subset V$ be an $\mathbb F$-subspace of $V$. The following conditions are equivalent.

(1) $W$ is $\hat K$-stable,

(2) $W$ is $\hat{\frak k}$-stable,

(3) $W$ is $\frak k$-stable.\end{lemma}
\begin{proof} Conditions (1) and (2) are equivalent for a connected group in characteristic 0. It is clear that (2) implies (3). We left to show that (3) implies (1).

Assume (3). Let $d$ be the dimension of $W$. Consider variety ${\rm Gr}(d; V)$ of $d$-dimensional subspaces of $V$ and denote by $[W]$ the point of ${\rm Gr}(d; V)$ defined by $W$. Let $H$ be the stabilizer of $[W]$ in $K$. Condition (3) implies that $\frak k$ is a subalgebra of the Lie algebra of $H$. The definition of $\hat K$ implies that $H=\hat K$. Therefore $W$ is a $\hat K$-submodule of $V$.\end{proof}
\begin{corollary}\label{Chats}Let $V$ be a finite-dimensional $\hat K$-module. Then the lattices of
\begin{center}$\hat K$-submodules of $V$, $\hat{\frak k}$-submodules of $V$ and $\frak k$-submodules of $V$\end{center}
coincide.\end{corollary}
\begin{corollary}\label{Cbhat}Let $M$ be a direct limit of finite-dimensional $\hat K$-modules. Then

(1) $M$ is a bounded $(\hat{\frak k}, \hat{\frak k})$-module if and only if $M$ is a bounded $(\frak k, \frak k)$-module,

(2) $M$ is an admissible $(\hat{\frak k}, \hat{\frak k})$-module if and only if $M$ is an admissible $(\frak k, \frak k)$-module.\end{corollary}
\begin{proof}This statement is implied by Lemma~\ref{Laev1} and Corollary~\ref{Chats}.\end{proof}






We wish to establish a connection between the admissible $(\frak g, \hat{\frak k})$-modules and the admissible $(\frak g, \frak k)$-modules. The first step is as follows.
\begin{lemma}\label{Ladm1}Let $M$ be an admissible $(\frak g, \frak k)$-module. Then $M$ is a $(\frak g, \hat{\frak k})$-module.\end{lemma}
\begin{proof} Let $M_0$ be a $\frak k$-stable subspace of $M$. Consider the $\hat{\frak k}$-submodule $\hat M_0$ of $M$ generated by $M_0$. It is enough to show that $\dim \hat M_0<\infty$. Set
$$M_{i+1}:=\hat{\frak k}M_i+M_i, i\ge0.$$
It is clear that each $M_i$ is finite dimensional and $\cup_iM_i=\hat M_0$. The definition of $M_i$ implies that there exists a surjective $\frak k$-morphism
$$\phi: (\hat{\frak k}/\frak k)\otimes M_i\twoheadrightarrow M_{i+1}/M_i$$
of $\frak k$-modules. Theorem~3 of~\cite[Subsection 3.3.3]{VO} implies that $$[\hat{\frak k}, \hat{\frak k}]=[\frak k, \frak k]=[\hat{\frak k}, \frak k]$$
and thus that $(\hat{\frak k}/\frak k)$ is a trivial $\frak k$-module. Therefore the list of simple $\frak k$-subquotients of $M_{i+1}$ equals the list of simple $\frak k$-subquotients of $M_{i}$ for all $i\ge0$, and hence these lists equal to the list of simple $\frak k$-subquotients of $M_0$. If $\lim\limits_{i\to\infty}\dim M_i=\infty$ then the multiplicity of at least one such a subquotient would tend to infinity with $i\to\infty$. That is incompatible with the assumption that $M$ is an admissible $(\frak g, \frak k)$-module.\end{proof}
Lemma~\ref{Ladm1} implies that if $M$ is an admissible $(\frak g, \frak k)$-module then

$\bullet$ $\Var(M)\subset(\hat{\frak k})^\bot$,

$\bullet$ $\Var(M)$ is stable under the action of $\hat K$.\\

\begin{proposition}\label{Pbhat} Let $\frak g$ be a semisimple Lie algebra, let $M$ be a simple $\frak g$-module, and let $\frak k\subset \frak g$ be a subalgebra. Then

(a) $M$ is a bounded $(\frak g, \hat{\frak k})$-module if and only if $M$ is a bounded $(\frak g, \frak k)$-module,

(b) $M$ is an admissible $(\frak g, \hat{\frak k})$-module if and only if  $M$ is an admissible $(\frak g, \frak k)$-module.\end{proposition}
\begin{proof} It is straightforward to argue that $M$ is a bounded $(\frak g, \hat{\frak k})$-module (respectively bounded $(\frak g, \frak k)$-module) if and only if $\mathbb F[\Var(M)]$ is a bounded $(\hat{\frak k}, \hat{\frak k})$-module (respectively bounded $(\frak k, \frak k)$-module), see~\cite[Proof of Proposition~2.6]{P1}. In the same way we have that $M$ is an admissible $(\frak g, \hat{\frak k})$-module (respectively admissible $(\frak g, \frak k)$-module) if and only if $\mathbb F[\Var(M)]$ is an admissible $(\hat{\frak k}, \hat{\frak k})$-module (respectively admissible $(\frak k, \frak k)$-module). Therefore it is enough to verify the following facts:

(a$'$) $\mathbb F[\Var(M)]$ is a bounded $(\hat{\frak k}, \hat{\frak k})$-module if and only if $\mathbb F[\Var(M)]$ is a bounded $(\frak k, \frak k)$-module.

(b$'$) $\mathbb F[\Var(M)]$ is an admissible $(\hat{\frak k}, \hat{\frak k})$-module if and only if $\mathbb F[\Var(M)]$ is an admissible $(\frak k, \frak k)$-module.\\
These statements are implied by Corollary~\ref{Cbhat}.\end{proof}
Let $L$ be a maximal connected reductive subgroup of $\hat K$. Such a subgroup is unique up to conjugacy. Denote by $\frak l\subset\frak g$ the Lie algebra of $L$. It is clear that $\frak l$ is algebraically reductive in $\frak g$. We conclude this section with the following. \begin{proposition}\label{Pbhatr}Let $M$ be a simple $(\frak g, \hat{\frak k})$-module. Then

(a) $M$ is an admissible $(\frak g, \hat{\frak k})$-module if and only if $M$ is an admissible $(\frak g, \frak l)$-module,

(b) $M$ is a bounded $(\frak g, \hat{\frak k})$-module if and only if $M$ is a bounded $(\frak g, \frak l)$-module.\end{proposition}
\begin{proof}This is implied by two facts:

(a) the restriction of a simple finite-dimensional $\hat K$-module to $L$ is simple,

(b) the restrictions of two nonisomorphic simple finite-dimensional $\hat K$-modules to $L$ are nonisomorphic.\end{proof}
\section{Proof of Theorem~\ref{Tmain}}
We use notation of Section~\ref{Smain} and of Theorem~\ref{Tmain}. Fix a simple $(\frak g, \frak k)$-module $M$. Then $\GVar(M)$ is irreducible~\cite{Jo}. Theorem~\ref{Tmain} is implied by the following propositions.
\begin{proposition}\label{Pmain1}If $M$ is a bounded $(\frak g, \frak k)$-module then $\GVar(M)$ is $K$-coisotropic.\end{proposition}
\begin{proposition}\label{Pmain2}If $\GVar(M)$ is $K$-coisotropic then $M$ is an admissible $(\frak g, \frak k)$-module.\end{proposition}
\begin{proposition}\label{Pmain3} If $M$ is an admissible $(\frak g, \frak k)$-module and $\GVar(M)$ is $K$-coisotropic then $M$ is a bounded $(\frak g, \frak k)$-module.\end{proposition}
\begin{proof}[Proof of Proposition~\ref{Pmain1}] Proposition 2.6 of~\cite{P1} implies that if $M$ is a bounded $(\frak g, \frak k)$-module then all irreducible components of $\Var(M)$ are $K$-spherical varieties.

The Gelfand-Kirillov dimension of $\Ue(\frak g)/\Ann_{\Ue(\frak g)}M$ equals $\dim\GVar(M)$. Theorem~9.11 of~\cite{KL} implies that the Gelfand-Kirillov dimension of $M$ is at least $\frac12\dim\GVar(M)$. Hence there exists an irreducible component $X$ of $\Var(M)$ such that $$\dim X\ge\frac12\dim\GVar(M).$$  Let $\mathcal O$ be the nilpotent coadjoint which is dense in $G\cdot X$. Proposition 2.6(b) of~\cite{P1} implies that $X\cap\mathcal O$ is Lagrangian in $\mathcal O$ and hence that $$\dim X=\frac12\dim\mathcal O\ge\frac12\dim\GVar(M).$$ On the other hand we have that $\mathcal O\subset\GVar(M)$. The variety $\GVar(M)$ is irreducible~\cite{Jo} and thus $$\GVar(M)=\overline{\mathcal O}.$$ Recall that $X\cap\mathcal O$ is Lagrangian in $\mathcal O$, and therefore $X\cap\mathcal O\subset\mathcal O$ is special in the sense of~\cite[Definition~4]{TZh}. Next, ~\cite[Theorem~7]{TZh} implies that
\begin{equation}2{\rm c}(X)={\rm cork} M+2\dim X-\dim \mathcal O\label{Etzh}\end{equation}
(notation of~\cite{TZh}), and thus $2{\rm c}(X)={\rm cork} M$. The variety $X$ is $K$-spherical and therefore ${\rm c}(X)=0$. Hence ${\rm cork} M=0$ and $M$ is $K$-coisotropic.\end{proof}
\begin{proof}[Proof of Proposition~\ref{Pmain2}] It is enough to show that $M$ is an admissible $(\frak g, \frak k)$-module, or equivalently that $\mathbb F[\Var(M)]$ is an admissible $(\frak k, \frak k)$-module, see~\cite[Proof of Proposition~2.6]{P1}. The last statement is equivalent to the condition $$\dim (\mathbb F[\Var(M)]^K)<\infty,$$
see~\cite[Theorem~3.24]{VP}.

The embedding $\frak k\hookrightarrow \frak g$ defines the sequence of maps
$$\Sa(\frak k)\hookrightarrow\Sa(\frak g)\twoheadrightarrow\mathbb F[\GVar(M)]\twoheadrightarrow\mathbb F[\Var(M)].$$ Further we have
$$\Sa(\frak k)^K\hookrightarrow\Sa(\frak g)^K\twoheadrightarrow\mathbb F[\GVar(M)]^K\twoheadrightarrow\mathbb F[\Var(M)]^K.$$
Next, we need the following lemma.
\begin{lemma}\label{Llo}For any maximal ideal $m\in\Sa(\frak k)^K$ the algebra $\mathbb F[\GVar(M)]^K/m\mathbb F[\GVar(M)]^K$ is finite dimensional.\end{lemma}
\begin{proof} The proof is based on~\cite[Theorem 1.2.1]{Lo}, see also~\cite[Corollary~5.9.1]{Lo}. We use the notation of~\cite{Lo}. 
The variety $\GVar(M)$ is irreducible~\cite{Jo} and we have
$${\rm trdeg}~\mathbb F(\GVar(M))^K=\dim \GVar(M)-m_K(\GVar(M))$$
where ${\rm trdeg}~\mathbb F(\GVar(M))^K$ is the transcendence degree of $\mathbb F(\GVar(M))^K$. This implies that the Gelfand-Kirillov dimension of $\mathbb F[X]^K$ is at most $\dim X-m_K(X)$ and therefore
$$\dim X/\!\!/K\le \dim X-m_K(X).$$
We fix a maximal ideal $m$ of $\Sa(\frak k)^K$ and denote by $[m]$ the respective point in the maximal spectrum $\frak k^*/\!\!/K$ of $\Sa(\frak k)^K$. The algebra $\mathbb F[\GVar(M)]/m\mathbb F[\GVar(M)]$ is finite dimensional if and only if the fiber of the map
$$X/\!\!/K\to \frak k^*/\!\!/K$$
is at most 0-dimensional. Theorem 1.2.1 of~\cite{Lo} implies that the dimension of this fiber is at most
$$\dim X/\!\!/K-{\rm\underline{def}}_KX\le\dim X-m_K(X)-{\rm\underline{def}}_KX.$$
The variety $\GVar(M)$ is generically symplectic in the sense of~\cite[Definition 2.2.3]{Lo} and therefore
$$\dim X=m_K(X)+{\rm\underline{def}}_KX,$$
see~\cite[Subsection 5.9]{Lo}. This implies that the dimension of the fiber under consideration is at most 0.\end{proof}
Put $m_0:=\Sa(\frak k)^K\cap(\frak k\Sa(\frak k))$. It is clear that $m_0$ is a maximal ideal of $\Sa(\frak k)^K$. Next,
$$\Var(M)\subset\frak k^\bot$$
implies that $\mathbb F[\Var(M)]$ is annihilated by $\frak k\Sa(\frak g)$. Therefore $\mathbb F[\Var(M)]^K$ is also annihilated by $m_0$. This implies that $\mathbb F[\Var(M)]^K$ is a quotient of
$$\mathbb F[\GVar(M)]^K/m_0\mathbb F[\GVar(M)]^K.$$
By Lemma~\ref{Llo} we have$$\dim \mathbb F[\GVar(M)]^K/m_0\mathbb F[\GVar(M)]^K<\infty,$$ and therefore $\dim \mathbb F[\Var(M)]^K<\infty$.
\end{proof}
\begin{proof}[Proof of Proposition~\ref{Pmain3}] Let $X_1,..., X_n$ be the irreducible components of $\Var(M)\subset\GVar(M)$. Proposition~2.6 of~\cite{P1} implies that it is enough to show that all $X_i$th are $K$-spherical varieties.

Fix $i$. The variety $G\cdot X_i$ contains unique dense orbit $\mathcal O_i$. Moreover, $X_i\cap \mathcal O_i$ is a coisotropic subvariety of $\mathcal O_i$~\cite{G}. On the other hand, $X_i$ is $K$-isotropic~\cite[Theorem~2]{P2}. This implies that $X_i$ is $K$-special in a sense of~\cite[Definition 4]{TZh}. Applying formula~(\ref{Etzh}) we have
$$2{\rm c}(X_i)={\rm cork}(\mathcal O_i).$$ Next, ${\rm c}(X_i)=0$ if and only if $X_i$ is $K$-spherical and thus we need to check that
$${\rm cork}~\mathcal O_i=0,$$ i.e. that the action of $K$ on $\mathcal O_i$ is coisotropic. This is a consequence of~\cite[Proposition 2.7]{AP} and the facts that $\mathcal O_i\subset\GVar(M)$ and $\GVar(M)$ is $K$-coisotropic, see also~\cite[Theorem~1.2.4]{Lo}.
\end{proof}
\section{Acknowledgements}
I would like to thank Roman Avdeev, Dmitry Timash\"ev and Vladimir Zhgun for useful discussions on spherical varieties. I also thank the math department of the University of Manchester for a very encouraging atmosphere. This work was supported by Leverhulme Trust Grant RPG-2013-293 and by RFBR grant 16-01-00818.

\end{document}